\documentclass[a4paper, reqno, 11pt]{article}
\usepackage{fullpage}
\usepackage{color}
\makeatletter

\@addtoreset{equation}{section}
\makeatother
\usepackage[singlespacing]{setspace}
\usepackage{xcolor}
\usepackage{here}
\usepackage{graphicx}
\usepackage{float}
\usepackage{mathtools}
\usepackage[T1]{fontenc}
\usepackage{amsmath}
\usepackage{amssymb}
\usepackage{latexsym}
\usepackage{amsthm}
\usepackage{hyperref}
\usepackage{tikz-cd}

\newtheorem{Thm}{Theorem}[section]

\newtheorem{Prop}[Thm]{Proposition}

\theoremstyle{definition}

\newtheorem{Rem}[Thm]{Remark}
\newtheorem{Exa}[Thm]{Example}
 
\def\GL{\mathrm{GL}}
\def\SL{\mathrm{SL}}
\def\SO{\mathrm{SO}}
\def\O{\mathrm{O}}
\def\Sym{\mathrm{Sym}}
\def\sym{\mathrm{sym}}
\def\SPD{\mathrm{SPD}}
\def\dx{{\mathrm{d}x}}

\def\FR{{\mathrm{FR}}}
\def\Ker{{\mathrm{Ker}}}

\begin{document}

\title{Group invariance of $f$-divergences and the Fisher--Rao distance}

\author{
Frank Nielsen\\
Sony Computer Science Laboratories Inc.\\
E-mail: {\tt Frank.Nielsen@acm.org} 
\and 
Kazuki Okamura\\
Department of Mathematics, Faculty of Science, Shizuoka University\\
E-mail:  {\tt okamura.kazuki@shizuoka.ac.jp}
}
\date{\today}

\maketitle

\begin{abstract}
Many statistical models have natural symmetries described by a group action. 
We study how such symmetries affect the comparison of two distributions. 
We work with a transformation model in which a group acts on both the sample space and the parameter space, and the densities transform with a multiplier. 
Under this assumption, we show that every $f$-divergence is invariant under the group action. As a consequence, an invariant divergence depends only on a maximal invariant of the pair of parameters. 
When the action on the parameter space is transitive, this maximal invariant is given by a double coset. 
We apply this result to multidimensional location-scale families, and we show that the same reduction applies to the Fisher--Rao distance.
\end{abstract}

\section{Introduction}

\subsection{Group actions and invariant divergences}
Many statistical models have natural symmetries. 
For example, in a location model, a common translation of the sample space and the parameter space should not change the intrinsic comparison between two distributions. 
In a scale model, a stretch in both the sample space and the parameter space should also not change such a comparison. 
These symmetries are described by group actions. 
If a statistical divergence (not necessarily metric) or a statistical distance is compatible with the model, then it should be invariant under the corresponding group action.

The purpose of this paper is to study this idea for $f$-divergences~\cite{Csiszar-1963} and for the Fisher--Rao distance~\cite{Rao-1945}. 
The class of $f$-divergences includes many important divergences, such as the Kullback--Leibler divergence, the Hellinger divergence, the total variation distance and the chi-square divergence. 
This class was introduced and studied independently in different forms by Ali and Silvey \cite{AliSilvey-1966}, Csisz\'{a}r \cite{Csiszar-1963}, and Morimoto \cite{Morimoto-1963}. 
Because of this generality, a result for all $f$-divergences gives a common explanation for many particular formulas.

We consider a statistical model which is equivariant under a group action~\cite{keener2010theoretical}. 
More precisely, the group acts on the sample space and also on the parameter space, and the densities transform with the correct multiplier. 
Under this assumption, every $f$-divergence between two members of the model is invariant. 
If both parameters are transformed by the same group element, then the value of the divergence does not change. 
Consequently, the divergence depends only on the relative position of the two parameters.

Invariance is useful, but it is not enough by itself. 
It tells us only that the divergence has the same value for all pairs of parameters that lie on the same orbit under the diagonal group action. 
In order to turn this observation into a formula, we need a quantity that labels these orbits. 
This is the role of a {\it maximal invariant}. 
It keeps exactly the information that remains after the group symmetry is removed, and it loses no information that can affect an invariant divergence. 
Therefore, once a maximal invariant for pairs of parameters is known, every invariant divergence must be expressible as a function of it. 
This viewpoint is classical in statistics; Eaton \cite{Eaton1989}  and Wijsman \cite{Wijsman1990} give a detailed account of invariant measures on groups and their statistical use.

The use of maximal invariants is especially important when we compare two distributions. 
A divergence is defined for a pair of parameters, not for a single parameter. 
Therefore, the relevant symmetry is the simultaneous action of the group on both parameters. 
Even if the action of the group on the parameter space $\Theta$ itself is transitive, the diagonal action on the product space $\Theta \times \Theta$ is usually not transitive. 
Some relative information remains, such as relative location, relative scale, or more generally the relative position of the two parameters. 
An invariant divergence necessarily depends only on the maximal invariant of the pair. In information geometry, a divergence is also interpreted geometrically as a contrast or yoke function~\cite{eguchi1992geometry}, i.e., a scalar function on the product manifold.

This observation gives a useful reduction principle: instead of treating a divergence as a function of all parameters, we may first remove the group symmetry and express it through a smaller set of invariant variables. 
This not only simplifies the computation, but also tells us which quantities can appear in any invariant formula. 

\subsection{Contributions}

For a transitive parameter space, the orbit space of pairs has a natural {\it double-coset} form. 
If the parameter space $\Theta$ is identified with a homogeneous space $G/H$, where $G$ is a group and $H$ is the stabilizer of a fixed base point, then a maximal invariant for pairs is given by the double coset $H g_1^{-1} g_2 H$ for $g_1, g_2 \in G$ (Theorem~\ref{thm:double-coset}). 
This gives a general and geometric way to express invariant divergences. 
The double coset represents the relative position of two points in the homogeneous space. 
This double coset is well-defined: it does not depend on the particular choices of $g_1$ and $g_2$ used to represent $\theta_1$ and $\theta_2$. 
Moreover, two pairs of parameters have the same double coset if and only if they are related by the diagonal action of $G$. 
Therefore, the double coset gives a complete description of the relative position of two parameters.

We apply this idea to location-scale families (see~\cite{NO2024} for a restricted study on either location or scale families). 
In these models, the relevant information is the relative location and relative scale of two distributions. 
In order to make this description canonical in higher dimensions, we use a quotient parametrization that ignores orthogonal rotations of the scale factor.

In this setting, the maximal invariant is described by the singular values of the relative scale and by the block norms of the transformed relative location, where the blocks are determined by the multiplicities of the singular values (Section~\ref{sec:lsmodel}).  
Thus, every invariant $f$-divergence necessarily depends only on these quantities. The same argument applies to the Fisher--Rao distance: affine invariance reduces the distance between two distributions to the distance from a canonical base distribution (Proposition~\ref{prop:FR-quotient}). 
This reduction does not give a general closed-form formula, but it tells us exactly which invariant variables are relevant. 

\subsection{Outline}
This paper is organized as follows. 
First, we prove the group invariance of $f$-divergences in a general transformation model (Theorem~\ref{thm:inv-div}). 
We then show that such divergences are functions of maximal invariants of the diagonal action on parameter pairs (Theorem~\ref{thm:abstract}). 
In the transitive case, we describe the maximal invariant by a double-coset space (Theorem~\ref{thm:double-coset}). 
Next, we construct explicit maximal invariants for multidimensional location-scale models in \S\ref{sec:lsmodel}. 
In \S\ref{sec:bdfy}, we consider Bregman divergences which are canonical dually flat divergences in information geometry~\cite{IG-2016} and describe the invariance in terms of dual group actions for the centered scale family.
Finally, we discuss the Fisher--Rao distance~\cite{Rao-1945} in \S\ref{sec:FRdist} and show that the same invariant reduction applies to this Riemannian  
geodesic distance (Proposition~\ref{prop:FR-quotient}).
 
\section{Expression of \texorpdfstring{$f$}{f}-divergences by maximal invariants}

We follow the framework of \cite[Section 3]{Eaton1989}.  
Let $G$ be a topological group. 
Let $(X, \mathcal{B}, \lambda)$ be a $\sigma$-finite measure space. 
Assume that $G$ acts on $X$ measurably, that is, a map $L_g$ on $X$ defined by $L_g (x) \coloneqq   gx$ is measurable. 
Assume that there is a multiplier $\chi$ on $G$ such that 
\begin{equation}\label{eq:rel-inv-mu}
\lambda(gB) = \chi(g) \lambda(B), \quad B \in \mathcal{B}, \quad g \in G. 
\end{equation} 
Let $\Theta$ be a set. 
Assume that $G$ also acts on $\Theta$. 
Let $\{p(\cdot | \theta)\colon\theta \in \Theta\}$ be a family of densities with respect to $\lambda$. 
Let $P_{\theta}(dx) \coloneqq   p(x|\theta)\lambda(dx)$. 
Assume that for every $\theta \in \Theta$, $p(x|\theta) > 0$, $\lambda$-a.e. $x$. 
Assume that 
\begin{equation}\label{eq:inv-density} 
\chi(g) p(gx | g\theta) = p(x|\theta), \ \textup{ $\lambda$-a.e. } x \in X, \quad \theta \in \Theta, \quad g \in G.   
\end{equation}
Let $f$ be a function on $(0,\infty)$ such that $f(1) = 0$ and $f$ is convex on $(0,\infty)$ (and strictly convex at $1$) and 
\[ \int_{X} \left|f \left( \frac{p(x|\theta_2)}{p(x|\theta_1)} \right)\right| p(x|\theta_1) \lambda(dx) < \infty, \quad  \theta_1, \theta_2 \in \Theta. \]
Let the $f$-divergence between $P_{\theta_1}$ and $P_{\theta_2}$ be 
\[ D_f (\theta_1\colon\theta_2) \coloneqq    \int_{X} f \left( \frac{p(x|\theta_2)}{p(x|\theta_1)} \right) p(x|\theta_1) \lambda(dx), \quad  \theta_1, \theta_2 \in \Theta. \]

\begin{Thm}[Group invariance in a transformation model]\label{thm:inv-div}
\[ D_f (g\theta_1 \colon g\theta_2)  = D_f (\theta_1 \colon \theta_2)  \ \  \theta_1, \theta_2 \in \Theta, \ g \in G. \]
\end{Thm}

\begin{proof}
Let 
\[ F(x) \coloneqq   f \left( \frac{p(x|g\theta_2)}{p(x|g\theta_1)} \right) p(x|g\theta_1).  \]
Then, by \eqref{eq:rel-inv-mu}, 
\[ \chi(g) \int_{X} F(gx) \lambda(dx) = \int_{X} F(x) \lambda(dx) =  D_f (g\theta_1 \colon g\theta_2). \]
On the other hand, by \eqref{eq:inv-density}, 
\[ \chi(g) \int_{X} F(gx) \lambda(dx) =  D_f (\theta_1 \colon \theta_2). \]
Thus, we have the assertion. 
\end{proof}

The group $G$ also acts diagonally on the product space $\Theta \times \Theta$ by $(g, (\theta_1, \theta_2)) \mapsto (g\theta_1, g\theta_2)$. 
Let the orbit of $(\theta_1, \theta_2) \in \Theta \times \Theta$ be $O_{(\theta_1, \theta_2)} \coloneqq   \{g(\theta_1, \theta_2) | g \in G\}$. 
If  the action of $G$ on $\Theta \times \Theta$ is transitive, 
then  $O_{(\theta_1, \theta_2)} = \Theta \times \Theta$ and $D_f (\theta_1 \colon \theta_2) = 0$ for every $(\theta_1, \theta_2) \in \Theta \times \Theta$.  
However, even if the action of $G$ on $\Theta$ is transitive,  the action of $G$ on $\Theta \times \Theta$ can be non-transitive.

\begin{Thm}[Maximal invariant]\label{thm:abstract}
Let $m = m(\theta_1, \theta_2)$ be a self-map on $\Theta \times \Theta$ such that 
$m(\theta_1, \theta_2) \in O_{(\theta_1, \theta_2)}$ for every $(\theta_1, \theta_2) \in \Theta \times \Theta$. 
Assume that $m(g\theta_1, g\theta_2) = m(\theta_1, \theta_2)$ for every $g \in G$ and $\theta_1, \theta_2 \in \Theta$. 
Then the map $m$ is a maximal invariant. 
In particular, 
$D_f (\theta_1 \colon \theta_2) $ is a function of $m(\theta_1, \theta_2)$. 
\end{Thm}

\begin{proof}
If $m(\theta_1, \theta_2) = m(\theta_1^{\prime}, \theta_2^{\prime})$, then 
$O_{(\theta_1, \theta_2)} \cap O_{(\theta_1^{\prime}, \theta_2^{\prime})} \ne \emptyset$. 
Hence, 
$O_{(\theta_1, \theta_2)} = O_{(\theta_1^{\prime}, \theta_2^{\prime})}$. 
Since $(\theta_1^{\prime}, \theta_2^{\prime}) \in  O_{(\theta_1^{\prime}, \theta_2^{\prime})} = O_{(\theta_1, \theta_2)}$, 
there exists $g \in G$ such that $(\theta_1^{\prime}, \theta_2^{\prime}) = g(\theta_1, \theta_2)$. 
\end{proof}

We remark that, in general, such a self-map exists due to {\it the axiom of choice}. 
It is important to give explicit constructions of the map.  

Assume that $G$ acts on $\Theta$ transitively. 
Let $H$ be the stabilizer of $\theta_0 \in \Theta$. 
Specifically, $H \coloneqq  \{h \in G \colon h\theta_0 = \theta_0 \}$ for some $\theta_0 \in \Theta$. 
This is a subgroup of $G$. 
The group structure of $H$ does not depend on the choice of $\theta_0$, since the action is transitive. 
Define the set of left cosets by $G/H \coloneqq  \{gH \colon g \in G\}$. 

Define a map $\phi \colon G/H \to \Theta$ by $\phi(gH) \coloneqq  g\theta_0$. 
This is well-defined and is a bijection. 
We identify $\Theta$ as $G/H$ by the map $\phi$. 
Define the {\it double coset space} by $H \backslash G / H \coloneqq   \{HgH \colon g \in G \}$.

\begin{Thm}[Double coset maximal invariance for transitive actions]\label{thm:double-coset}
Define a map $m \colon \Theta \times \Theta \to H \backslash G / H$ by 
\[ m(g_1 H, g_2 H) \coloneqq  H g_1^{-1} g_2 H. \]
Then, the map $m$ is a maximal invariant. 
\end{Thm}

\begin{proof}
Assume that $g_i H = g_i^{\prime} H$, $i=1,2$. 
Then there exist $h_i \in H$, $i=1,2$, such that $g_i = g_i^{\prime} h_i$, $i=1,2$. 
Since $H h_1^{-1} = h_2 H = H$, 
\[ H g_1^{-1} g_2 H = H h_1^{-1} (g_1^{\prime})^{-1} g_2^{\prime} h_2 H = H (g_1^{\prime})^{-1} g_2^{\prime} H. \]
Hence, the map $m$ is well-defined. 
We see that 
\[ m(g(g_1 H), g(g_2 H)) = m((g g_1)H, (g g_2)H) = H (g g_1)^{-1} (g g_2) H\]
\[ = H g_1^{-1} g_2 H = m(g_1 H, g_2 H). \]
Assume that $m(g_1 H, g_2 H) = m(g_1^{\prime} H, g_2^{\prime} H)$. 
Then $H g_1^{-1} g_2 H = H (g_1^{\prime})^{-1} g_2^{\prime} H$. 
Then there exist $h_1, h_2 \in H$ such that $g_1^{-1} g_2 = h_1 (g_1^{\prime})^{-1} g_2^{\prime} h_2$. 
Let $g \coloneqq  g_1^{\prime} h_1^{-1} g_1^{-1}$. 
Then $g g_1 H = g_1^{\prime} H$ and $g g_2 H = g_2^{\prime} H$. 
\end{proof}

\section{Examples}

\subsection{Location-scale models}\label{sec:lsmodel}

We give an explicit construction of a maximal invariant in the location-scale model. 
Let $d \ge 1$. 
Let $\GL(d, \mathbb{R})$ be the general linear group on $\mathbb{R}^d$. 
Denote the unit matrix of size $d$ by $I_d$. 

\begin{Exa}\label{exa:naive}
We consider the multi-dimensional location-scale model in \cite[Section 3]{Eaton1989}. 
However, this is only a naive formulation, since the resulting object below fails to be a maximal invariant when $d\ge 2$.  
\begin{itemize}
\item (sample space) $(X, \mathcal{B}, \lambda) = (\mathbb{R}^d, \mathcal{B}(\mathbb{R}^d), \ell)$, where $\ell$ denotes the Lebesgue measure. 
\item (group) $G = \textup{Aff}(d) = \GL(d, \mathbb{R}) \times \mathbb{R}^d$, and the group operation is defined by $(A_1, b_1)(A_2, b_2) = (A_1 A_2, A_1 b_2 + b_1)$. 
The unit element is $(I_d, 0)$. 
\item (action on sample space) An action $G \curvearrowright X$ is defined by $((A,b), x) \mapsto Ax + b$. 
\item (multiplier) $\chi((A,b)) \coloneqq   |\det(A)|$ is a multiplier on $G$. 
Then by change of variables, \eqref{eq:rel-inv-mu} holds, that is, 
\[ \ell(A(B) + b) = \chi((A,b)) \ell(B), \quad B \in \mathcal{B}(\mathbb{R}^d), \]
where we let $A(B) + b \coloneqq  \{Ax+b \colon x \in B\}$.   
\item (base density) Let $f_0$ be a positive probability density function, that is, $f_0 > 0$ and $\displaystyle \int_{\mathbb{R}^d} f_0 (x) \ell(dx) = 1$. 
Assume also that $f_0$ is radial, that is, $f_0 (x) = f_1 (\|x\|^2)$ for some positive Borel measurable function $f_1$ on $[0,\infty)$. 
\item (parameter space) $\Theta = \mathbb{R}^d \times \SPD(d)$, 
where $\SPD(d)$ is the set of symmetric and positive-definite matrices. 
\item (parametric family)
\[ p\left(x | (\mu, \Sigma) \right) \coloneqq   \frac{1}{\sqrt{\det \Sigma}} f_0 \left(\Sigma^{-1/2}(x - \mu)\right), \ \ x \in \mathbb{R}^d. \]
This is a probability density function. 
\item  (action on parameter space) An action $G \curvearrowright \Theta$ is defined by $((A,b), (\mu, \Sigma)) \mapsto (A\mu + b, A \Sigma A^{\top})$. 
\end{itemize}

Since $f_0$ is radial, \eqref{eq:inv-density} holds, that is, 
\[ \chi((A,b)) p\left((A,b)x | (A,b)(\mu, \Sigma)\right) = p\left(x | (\mu, \Sigma) \right). \]

Then 
$D_f \left((\mu_1, \Sigma_1), (\mu_2, \Sigma_2)\right) $ is a function of $\left(\Sigma_2^{-1/2} (\mu_1 - \mu_2),  \Sigma_2^{-1/2} \Sigma_1^{1/2}\right)$, by the change of variables formula. 
See \cite[Proposition 1]{NO2024}. 

However, $\Sigma_2^{-1/2} \Sigma_1^{1/2}$ can be asymmetric and fail to be positive-definite if $d \ge 2$. 
See \cite[Remark 3.1.15]{BCR1984}. 
For $d = 1$, this is a maximal invariant. 
\end{Exa}

Now we modify the parameter space in order to obtain a maximal invariant.

\begin{Exa}\label{exa:modified}

\!

\begin{itemize}
\item (sample space) $(X, \mathcal{B}, \lambda) = (\mathbb{R}^d, \mathcal{B}(\mathbb{R}^d), \ell)$, where $\ell$ denotes the Lebesgue measure.  
\item (group) $G = \textup{Aff}(d) = \GL(d, \mathbb{R}) \times \mathbb{R}^d$, and the group operation is defined by $(A_1, b_1)(A_2, b_2) = (A_1 A_2, A_1 b_2 + b_1)$. 
The unit element is $(I_d, 0)$. 
\item (action on sample space) An action $G \curvearrowright X$ is defined by $((A,b), x) \mapsto Ax + b$. 
\item (modified parameter space) $\Theta = \mathbb{R}^d \times \GL(d, \mathbb{R})$. 
\item (base density) Let $f_0$ be a positive probability density, that is, $f_0 > 0$ and $\displaystyle \int_{\mathbb{R}^d} f_0 (x) \ell(dx) = 1$. 
Assume also that $f_0$ is radial, that is, $f_0 (x) = f_1 \left(\|x\|^2\right)$ for some positive Borel measurable function $f_1$ on $[0,\infty)$.
\item (parametric family) 
\[ p\left(x | (\mu, V) \right) \coloneqq   \frac{1}{|\det V|} f_0 \left(V^{-1}(x - \mu)\right), \ \ x \in \mathbb{R}^d. \]
This is a probability density function. 
\item (action on parameter space) An action $G \curvearrowright \Theta$ is defined by $((A,b), (\mu, V)) \mapsto (A\mu + b, A V)$. 
\item (multiplier) $\chi((A,b)) \coloneqq   |\det(A)|$ is a multiplier on $G$. 
Then by the change of variable formula, \eqref{eq:rel-inv-mu} holds, that is, 
\[ \ell(A(B) + b) = \chi((A,b)) \ell(B), \ B \in \mathcal{B}(\mathbb{R}^d). \] 
\end{itemize}

Since $f_0$ is radial, \eqref{eq:inv-density} holds, that is, 
\[ \chi((A,b)) \ p\left((A,b)x | (A,b)(\mu, V)\right) = p\left(x | (\mu, V)\right). \]

Consider a  map $m \colon \Theta \times \Theta \to \Theta$ defined by 
\[ m\left((\mu_1, V_1), (\mu_2, V_2)\right) \coloneqq   \left(V_2^{-1} (\mu_1 - \mu_2),  V_2^{-1} V_1 \right). \]

By changing variables, 
$D_f \left((\mu_1, V_1), (\mu_2, V_2)\right) $ is a function of $m\left((\mu_1, V_1), (\mu_2, V_2)\right)$. 
Furthermore, the map $m$ is a maximal invariant. 
Indeed, 
if $m\left((\mu_1, V_1), (\mu_2, V_2)\right) = m\left((\mu_1^{\prime}, V_1^{\prime}), (\mu_2^{\prime}, V_2^{\prime})\right)$, 
then  $V_1^{\prime} V_1^{-1} = V_2^{\prime} V_2^{-1}$ and $\mu_1^{\prime} - A \mu_1 = \mu_2^{\prime} - A \mu_2$. 
Let $A \coloneqq   V_2^{\prime} V_2^{-1}$ and $b = \mu_1^{\prime} - A \mu_1$. 
Then $(A,b) (\mu_i, V_i) = (\mu_i^{\prime}, V_i^{\prime})$, $i=1,2$. 

However, this model {\it may} not satisfy the identifiability property. 
We need to modify the definition of $\Theta$ by introducing an equivalence relation. 

By changing variables, the characteristic function $\varphi_{\mu, V}$ of the density $p(x | (\mu, V) )$ satisfies the following. 
\[ \varphi_{\mu, V}(\xi) = \exp\left(i \langle \xi, \mu\rangle\right) \varphi_{0, I_d} \left(V^{\top} \xi\right), \ \xi \in \mathbb{R}^d.  \]
Since $f_0$ is radial, $\varphi_{0, I_d} $ is also radial. 
In particular, $\varphi_{0, I_d}(\xi) = \varphi_{0, I_d}(-\xi)$ for every $\xi \in \mathbb{R}^d$, and therefore $\varphi_{0, I_d}$ is real-valued. 
Since $f_0$ is strictly positive, 
$\left| \varphi_{0, I_d} (\xi) \right| < 1$ for every $\xi \ne 0$. 
By the Riemann--Lebesgue lemma, 
$\displaystyle \lim_{\|\xi\| \to \infty} \varphi_{0, I_d} (\xi) = 0$.

We also assume that 
\begin{equation}\label{eq:second-integrable}
\int_{\mathbb{R}^d} \|x\|^2 f_0 (x) \ell(dx) < \infty. 
\end{equation}

By considering the first and second derivatives of 
\[ \Phi(t) \coloneqq   \varphi_{0, I_d} \left((t,0,\dots,0) \right) = \int_{\mathbb{R}^d} \cos(tx_1) f_0 (x) \ell(dx) \] 
in a neighborhood of $0$, 
$\Phi$ is strictly concave around $t = 0$. 
By this and the fact that $|\Phi(t)| < 1$ for every $t \ne 0$, 
there exists $\epsilon_0 \in (0,1/2)$ such that for every $t_0 \in (1-\epsilon_0,1)$ there exists a unique $t_1 \in (0,\infty)$ such that $\Phi(t_1) = t_0$.

Since $\varphi_{0,I_d}(0) = 1$ and $\varphi_{0,I_d}$ is continuous, there exists
a neighborhood $U$ of $0$ such that
$\varphi_{0,I_d}(V_i^\top \xi) > 0$ for $i=1,2$ and $\xi \in U$.
Hence, for $\xi \in U$,
\[ \exp(i\langle \xi,\mu_1 - \mu_2 \rangle) = \frac{\varphi_{0,I_d}(V_2^\top\xi)}{\varphi_{0,I_d}(V_1^\top\xi)} \in (0,\infty). \]
Since the left-hand side has modulus one, it must be equal to one. Thus
$\exp(i\langle \xi,\mu_1 - \mu_2\rangle) = 1$ for all $\xi$ in a neighborhood of $0$,
and consequently $\mu_1 = \mu_2$.  
Hence, 
\[   \Phi\left( \| V_1^{\top} \xi \|\right) = \varphi_{0, I_d} \left(V_1^{\top} \xi\right) = \varphi_{0, I_d} \left(V_2^{\top} \xi\right) = \Phi\left(\| V_2^{\top} \xi \|\right), \quad \xi \in \mathbb{R}^d. \]
Hence, there exists $\epsilon_1 > 0$ such that for every $\xi$ such that $0 < \|\xi\| < \epsilon_1$, 
$\Phi(\| V_1^{\top} \xi \|) = \Phi(\| V_2^{\top} \xi \|) \in (1-\epsilon_0, 1)$ and furthermore, $\| V_1^{\top} \xi \| = \| V_2^{\top} \xi \|$. 
Hence, for every $\xi \in \mathbb{R}^d$, 
$\| V_1^{\top} \xi \| = \| V_2^{\top} \xi \|$. 
By considering the second-order partial derivatives of $\|V_i^{\top} \xi\|^2$ with respect to $\xi$, $i=1,2$, 
we see that $V_1 V_1^{\top} = V_2 V_2^{\top}$. 

Conversely, 
if $V_1 V_1^{\top} = V_2 V_2^{\top}$ holds, then 
there exists $P \in \O(d, \mathbb{R})$ such that $V_1 = V_2 P$.  
Since $\varphi_{0, I_d} $ is radial, 
if additionally $\mu_1 = \mu_2$, 
$\varphi_{\mu_1, V_1}(\xi) = \varphi_{\mu_2, V_2}(\xi)$.

Define an equivalence relation on $\GL(d, \mathbb{R})$ by 
$A \sim B$ if and only if there exists $P \in \O(d, \mathbb{R})$ such that $A = BP$. 
We denote the quotient space with respect to this relation by $\GL(d, \mathbb{R}) / \O(d,\mathbb{R})$ and the equivalence class of $V \in \GL(d,\mathbb{R})$ by $[V]$. 
It is a homogeneous space. 

If we replace $\GL(d, \mathbb{R})$ with the quotient space $\GL(d, \mathbb{R}) / \O(d,\mathbb{R})$, 
then the identifiability property holds for $\Theta = \mathbb{R}^d \times \GL(d, \mathbb{R}) / \O(d,\mathbb{R})$.

Let $\widetilde\Theta \coloneqq  \mathbb{R}^d \times \GL(d, \mathbb{R}) / \O(d,\mathbb{R})$. 
This is a modified parameter space.  
The action of $G = \textup{Aff}(d)$ on $\widetilde\Theta$ is given by
\[ ((A,b),(\mu,[V])) \mapsto (A\mu+b,[AV]). \]
Let $\theta_0 \coloneqq  (0,[I_d]) \in \widetilde\Theta$ and let the stabilizer 
\[ H \coloneqq  \mathrm{Stab}_G(\theta_0) = \left\{(P,0)\in \textup{Aff}(d): P \in \O(d,\mathbb{R}) \right\}. \]
Then the action of $G$ on $\widetilde{\Theta}$ is transitive and we can identify $\widetilde{\Theta}$ with $G/H$.
Hence, by Theorem \ref{thm:double-coset}, 
the map $\widetilde m  \colon \widetilde\Theta \times \widetilde\Theta \to H \backslash G/H$
defined by
\[ \widetilde m(\theta_1,\theta_2) \coloneqq  H g_2^{-1}g_1 H, \ \theta_i = g_i \theta_0, \ i = 1,2, \]
is a maximal invariant\footnote{In the following we use the reversed convention $H g_2^{-1}g_1 H$, which is also a maximal invariant and is convenient for comparison with $V_2^{-1} V_1$.}.
In particular, for $\theta_{i} = \left(\mu_i, [V_i] \right)$ we may choose representatives
$V_i\in \GL(d,\mathbb{R})$ and let $g_i\coloneqq  (V_i,\mu_i)\in \textup{Aff}(d)$. 
Hence, 
\[ g_2^{-1}g_1 = \left(V_2^{-1}V_1,\; V_2^{-1}(\mu_1-\mu_2)\right). \]
Therefore, 
\[ \widetilde m\left((\mu_1,[V_1]),(\mu_2,[V_2]) \right) = H\left(V_2^{-1}V_1,\; V_2^{-1}(\mu_1-\mu_2)\right)H. \]
Consequently, $D_{f} \left((\mu_1,[V_1]):(\mu_2,[V_2]) \right)$ is a function of
$\widetilde m\left((\mu_1,[V_1]),(\mu_2,[V_2]) \right)$.

We can further parameterize the double coset by a singular value decomposition.
Set
$S \coloneqq  V_2^{-1}V_1\in \GL(d,\mathbb{R}), \  \nu \coloneqq  V_2^{-1}(\mu_1-\mu_2) \in \mathbb{R}^d$. 
Take a singular value decomposition $S = U\Sigma W^\top$ with
$U, W \in \O(d,\mathbb{R})$ and
\[ \Sigma = \mathrm{diag}(\sigma_1,\dots,\sigma_d), \  \ \sigma_1 \ge \cdots \ge \sigma_d >0. \]
Then
\[ \left(U^\top,0\right)\,(S,\nu)\,(W,0) = \left(\Sigma,\;U^\top \nu \right). \]

We remark that the singular value decomposition is {\it not} unique. 
If $Q\in \O(d,\mathbb{R})$ satisfies $Q\Sigma = \Sigma Q$, 

then $S = (UQ)\Sigma(WQ)^\top$ is another singular value decomposition  of $S$. 
Conversely, if $S = U_i \Sigma W_i^\top$ and $U_i, W_i \in \O(d,\mathbb{R})$ for $i=1,2$, then 
there exists $Q \in \O(d,\mathbb{R})$ such that $Q\Sigma = \Sigma Q$, $U_2 = U_1 Q$ and $W_2 = W_1 Q$. 

Therefore, the map $z \coloneqq  U^\top \nu$ is determined only up to the action
$z\mapsto Q^\top z$ satisfying $Q\Sigma = \Sigma Q$.

Write 
\[ \Sigma = \mathrm{diag}\left(\tau_1 I_{m_1},\dots,\tau_k I_{m_k} \right), \ \tau_1 > \cdots > \tau_k > 0, \quad \sum_{j=1}^k m_j = d, \]
and decompose $z = (z^{(1)},\dots,z^{(k)}) \in \mathbb{R}^{m_1} \times \cdots \times \mathbb{R}^{m_k}$ accordingly. 
Then $Q\Sigma = \Sigma Q$ if and only if $Q = \mathrm{diag}(Q_1,\dots,Q_k)$ with
$Q_j\in \O(m_j, \mathbb{R})$, and therefore $z = (z^{(1)},\dots,z^{(k)})$ is determined up to
$z^{(j)} \mapsto Q_j^\top z^{(j)}$.

Thus, the orbit of $z$ is determined by the block norms
$r_j \coloneqq  \left\| z^{(j)} \right\|, \  j = 1,\dots,k$. 
Therefore, $D_{f} \left((\mu_1,[V_1]):(\mu_2,[V_2]) \right)$ depends only on
the singular values of $V_2^{-1}V_1$ and the block norms $(r_1,\dots,r_k)$ determined by $V_2^{-1}(\mu_1 - \mu_2)$ and the multiplicity of the singular values of $V_2^{-1}V_1$.
\end{Exa}

The preceding construction gives the following canonical form of the invariant reduction.

\begin{Prop}
Suppose the setting of Example \ref{exa:modified} with \eqref{eq:second-integrable}. 
Let 
$S \coloneqq  V_2^{-1}V_1$ and $\nu \coloneqq  V_2^{-1}(\mu_1-\mu_2)$. 
Let $S = U\Sigma W^\top$ be a singular value decomposition of $S$, where
\[ \Sigma = \operatorname{diag}(\tau_1 I_{m_1},\ldots,\tau_k I_{m_k}), \qquad \tau_1 > \cdots > \tau_k >0. \]
Write
$U^{\top} \nu \eqqcolon \left(z^{(1)},\ldots,z^{(k)} \right)$ 
according to the same block decomposition, and define
$r_j \coloneqq \|z^{(j)}\|,\quad j=1,\ldots,k$.
Then the maximal invariant is parametrized by the singular values $(\tau_1, \dots, \tau_k)$ of $S$, with their multiplicities $(m_1, \dots, m_k)$, together with the block norms
$(r_1, \dots, r_k)$. 
Hence $D_f\left((\mu_1,[V_1]):(\mu_2,[V_2])\right)$ depends only on these quantities.
\end{Prop}

 \subsection{Dual group actions for dually flat  Bregman/Fenchel-Young divergences}\label{sec:bdfy}

\begin{Exa}\label{exa:dual-group-action}
We consider a centered matrix scale family and a dually flat divergence~\cite{ohara1996dualistic,IG-2016} (i.e., Bregman divergence or Fenchel-Young divergence when coordinatized) on the same scale
parameter space.

\begin{itemize}
\item (sample space)  $(X,\mathcal{B},\lambda)=(\mathbb{R}^d,\mathcal{B}(\mathbb{R}^d),\ell)$, where
$\ell$ denotes the Lebesgue measure.

\item (group)  $G=GL(d,\mathbb{R})$, with group operation given by matrix multiplication.

\item (action on sample space)  An action $G\curvearrowright X$ is defined by $(A,x)\mapsto Ax$.

\item (multiplier)  $\chi(A)\coloneqq |\det A|$. Then
\[ \ell(A(B))=\chi(A)\ell(B), \quad B\in \mathcal{B}(\mathbb{R}^d). \]

\item (base density)  Let $f_0$ be a positive probability density on $\mathbb{R}^d$, and assume that
$f_0$ is radial, that is, $f_0(x) = f_1 \left(\|x\|^2 \right)$
for some positive Borel measurable function $f_1$ on $[0,\infty)$.

\item (parameter space for the scale family)  $\Theta = \SPD(d)$.

\item (parametric family)  For $\Sigma\in \SPD(d)$, define
\[ p(x\mid \Sigma)\coloneqq \frac{1}{\sqrt{\det\Sigma}}\,f_0(\Sigma^{-1/2}x), \quad x\in\mathbb{R}^d.\]
This is a probability density on $\mathbb{R}^d$.

\item (group action on parameter space) An action $G\curvearrowright \Theta$ is defined by $(A,\Sigma) \mapsto A\Sigma A^{\top}$.

\item (dual coordinate spaces)  Let $\mathcal{N} \coloneqq -\SPD(d)$ and $\mathcal{E}\coloneqq \SPD(d)$. 

\item (potential function)  Fix $\alpha>0$ and define
\[ F(\theta) \coloneqq -\alpha\log\det(-\theta), \quad \theta\in\mathcal{N}. \]
Then $F$ is a Legendre-type convex function and
$\eta = \nabla F(\theta) = -\alpha\theta^{-1}\in\mathcal{E}$. 

\item (identification with the scale parameter)  We identify $\Sigma\in \SPD(d)$ with
\[ \theta = -\alpha\Sigma^{-1},\quad \eta=\Sigma.\]

\item (dual group actions)  The congruence action on $\Sigma$ induces dual actions on $(\theta,\eta)$:
\[ A\circ\theta\coloneqq A^{-\top}\theta A^{-1},\quad A\bullet\eta\coloneqq A\eta A^\top. \]

\item (pairing)  The pairing between $\mathcal{N}$ and $\mathcal{E}$ is given by
\[ \langle X,Y\rangle\coloneqq \operatorname{tr}(XY), \quad X \in \mathcal{N}, Y \in \mathcal{E}. \]
\end{itemize}

Since $f_0$ is radial, \eqref{eq:inv-density} holds, that is, 
\[ \chi(A)\,p(Ax\mid A\Sigma A^\top) = p(x\mid\Sigma).\]
Hence, by Theorem \ref{thm:inv-div},
\[ D_{f} (A \Sigma_{1} A^\top : A \Sigma_{2} A^\top) = D_{f}(\Sigma_1 : \Sigma_2). \]

Moreover, the action of $G$ on $\SPD(d)$ is transitive. 
With $\Sigma_0 = I_d$, the stabilizer
is $H=O(d,\mathbb{R})$, and a maximal invariant of $(\Sigma_1,\Sigma_2)$ is
\[ m(\Sigma_1,\Sigma_2) = HA_2^{-1}A_1H,\quad \Sigma_i=A_iA_i^\top. \]
Equivalently, $m(\Sigma_1,\Sigma_2)$ is determined by the eigenvalues of
$\Sigma_2^{-1}\Sigma_1$. Therefore $D_f(\Sigma_1:\Sigma_2)$ is a function of this spectrum.

The convex conjugate of $F$ is
\[ F^{*} (\eta) = -\alpha\log\det\eta + \alpha d(\log\alpha-1),\quad \eta\in \SPD(d). \]
Furthermore,
\[ \nabla F(A\circ\theta)=A\bullet\nabla F(\theta), \quad \langle A\circ\theta,A\bullet\eta\rangle = \langle\theta,\eta\rangle, \]
and
\[ F(A\circ\theta) = F(\theta)+2\alpha\log|\det A|, \quad F^*(A\bullet\eta) = F^*(\eta)-2\alpha\log|\det A|.\]
Hence the Fenchel--Young divergence~\cite{Acharyya2013LearningToRank}
\[ Y_{F,F^*}(\theta,\eta')\coloneqq F(\theta)+F^*(\eta')-\langle\theta,\eta'\rangle \]
is invariant under the dual actions denoted by $\circ$ and $\bullet$:
\[ Y_{F,F^*}(A\circ\theta,A\bullet\eta')=Y_{F,F^*}(\theta,\eta'). \]

Since
\[ B_F(\theta_1:\theta_2) = F(\theta_1)-F(\theta_2)-\langle \nabla F(\theta_2),\theta_1-\theta_2\rangle = Y_{F,F^*}(\theta_1,\eta_2), \qquad \eta_2=\nabla F(\theta_2),\]
the associated Bregman divergence also satisfies
\[ B_F (A\circ\theta_1:A\circ\theta_2) = B_F(\theta_1:\theta_2).\]

In terms of the scale matrices $\Sigma_i$, $i=1,2$, this becomes
\[ B_F(\theta_1:\theta_2) = \alpha\left(\operatorname{tr}(\Sigma_2\Sigma_1^{-1}) -\log\det(\Sigma_2\Sigma_1^{-1})-d
\right), \qquad \theta_i = -\alpha\Sigma_i^{-1}, \ i = 1,2. \]
Hence, $B_F (\theta_1:\theta_2)$ is also a function of the eigenvalues of $\Sigma_2^{-1}\Sigma_1$. 
Such a divergence is called a {\it spectral divergence}.

Thus the same group $\GL(d,\mathbb{R})$ acts on the statistical scale family and on its dual coordinate systems. 
The $f$-divergence invariance follows from the transformation model, while the dually flat divergence invariance follows from the cancellation between the dual actions. 
When the scale family is an exponential family and its Kullback--Leibler divergence is represented by this Bregman divergence, the two viewpoints coincide. 
\end{Exa}

\begin{Rem}[Relation to $V$-potential geometry]
The log-determinant potential used in Example~\ref{exa:dual-group-action} is closely related to the
$V$-potential geometry on the cone of positive definite matrices studied by Ohara and Eguchi \cite{ohara-eguchi, ohara-eguchi-2014}.  
In both settings, a potential depending only on the determinant of a positive definite matrix induces a dually flat structure, and the resulting geometry is compatible with congruence actions on positive definite matrices.
In particular, the choice $V(s) = -\alpha \log s$ corresponds to the log-determinant case appearing in Example \ref{exa:dual-group-action}.  
Thus, Example \ref{exa:dual-group-action} shares with the $V$-potential framework the same basic mechanism, specifically, a determinant-type potential on positive definite matrices and its compatibility with matrix congruence actions.
\end{Rem}

\section{Fisher--Rao distance}\label{sec:FRdist}

The Fisher--Rao distance is the Riemannian (geodesic) distance of the Fisher metric. 
Let $\{p_{\theta} (x) \colon \theta \in \Theta\}$ be a smooth statistical model on a measure space $(X, \mathcal{B}, \lambda)$. 
Assume that $\Theta$ is a smooth manifold and $\theta \mapsto p_{\theta}(x)$ is positive and smooth for every $x \in X$. 
Let $g^{F} = \left\{g^{F}_{ij} \right\}_{i,j}$ be the Fisher metric, specifically, 
\[ g^{F}_{ij}(\theta) = \int_X p_{\theta}(x) \, \frac{\partial}{\partial \theta_i} \log p_{\theta}(x)  \frac{\partial}{\partial \theta_j} \log p_{\theta}(x) \, \lambda(\dx), \ \theta = (\theta_i)_i \in \Theta.  \]
Then the Fisher--Rao distance $d_{\FR}(\theta, \theta^{\prime})$ is defined by 
\[ d_{\FR}(\theta, \theta^{\prime}) \coloneqq  \inf_{\gamma} \int_0^1 \sqrt{g^{F}_{\gamma(t)} (\dot{\gamma}(t), \dot{\gamma}(t))} \ dt,  \quad \theta, \theta^{\prime} \in \Theta,   \]
where the infimum is taken over all piecewise differentiable curves $\gamma \colon [0,1] \to \Theta$ satisfying $\gamma(0) = \theta$ and $\gamma(1) = \theta^{\prime}$.

We consider the framework of Example \ref{exa:naive}. 
We further assume that the function $f_1$ is in $C^2 ([0,\infty))$ and positive and Borel measurable. 
Then $f_0 \in C^2 (\mathbb{R}^d)$ and the  map $(\mu, \Sigma) \mapsto \log p\left(x | (\mu, \Sigma) \right)$ is in $C^2 \left(\mathbb{R}^d \times \SPD(d)\right)$ for every $x \in \mathbb{R}^d$. 

Let 
\[ a_1 \coloneqq  \int_{\mathbb{R}^d} \left( \frac{ f_1^{\prime} \left(\|x\|^2 \right)}{ f_1 \left(\|x\|^2 \right)}\right)^2 \|x\|^2 f_1 \left(\|x\|^2 \right) \, \ell(dx)  \]
and 
\[ a_2 \coloneqq  \int_{\mathbb{R}^d} \left( \frac{ f_1^{\prime} \left(\|x\|^2 \right)}{ f_1 \left(\|x\|^2 \right)}\right)^2 \|x\|^4 f_1 \left(\|x\|^2 \right) \, \ell(dx). \]
Assume that $a_1$ and $a_2$ are finite. 

The tangent space $T_{(\mu, \Sigma)} (\mathbb{R}^d \times \SPD(d))$ is isomorphic to $\mathbb{R}^d \times \Sym(d)$, where $\Sym(d)$ denotes the set of symmetric matrices of order $d$.  
We see that for $(u,U), (w,W) \in \mathbb{R}^d \times \Sym(d)$, 
\[ g_{(\mu, \Sigma)} \left( (u,U), (w,W)\right) \]
\[ = \frac{4a_1}{d} u^{\top} \Sigma^{-1} w + \frac{2 a_2}{d(d+2)} \textup{tr}(\Sigma^{-1}U\Sigma^{-1}W) + \left(\frac{a_2}{d(d+2)}  - \frac{1}{4} \right) \textup{tr}(\Sigma^{-1}U) \textup{tr}(\Sigma^{-1}W). \]

Denote the action of the affine group $\textup{Aff}(d)$ on $\Theta$ by 
\[ \Phi_{(A,b)}(\mu, \Sigma) \coloneqq  (A\mu + b, A \Sigma A^{\top}). \]
Then the derivative of the map $\Phi_{(A,b)}$ is a linear transformation on the tangent space $\mathbb{R}^d \times \Sym(d)$ and is  given by 
\[ d\Phi_{(A,b)} (u,U) = \left(Au, AUA^{\top} \right), \ (u,U) \in \mathbb{R}^d \times \Sym(d). \]
Therefore, 
\[ g_{\Phi_{(A,b)} (\mu, \Sigma)} \left(  d\Phi_{(A,b)}(u,U),  d\Phi_{(A,b)}(w,W)\right) = g_{(\mu, \Sigma)} \left( (u,U), (w,W)\right). \]

If $(A,b) = \left(\Sigma_2^{-1/2}, - \Sigma_2^{-1/2} \mu_2 \right)$, then \[ \Phi_{(A,b)}(\mu_1, \Sigma_1) = \left(\Sigma_2^{-1/2}(\mu_1 - \mu_2),  \Sigma_2^{-1/2} \Sigma_1 \Sigma_2^{-1/2} \right)\] 
and $\Phi_{(A,b)}(\mu_2, \Sigma_2) = (0, I_d)$. 

Thus, we have similar results to those in Example \ref{exa:naive}.  
Specifically, we have the following: 
\begin{Prop}
If the parameter space is $\mathbb{R}^d \times \SPD(d)$, then we have the following:\\
(i) 
\[ d_{\FR}\left((A\mu_1 + b, A\Sigma_1 A^{\top}), (A\mu_2 + b,  A\Sigma_2 A^{\top}) \right) = d_{\FR} \left((\mu_1, \Sigma_1), (\mu_2, \Sigma_2)\right). \]
(ii) 
$d_{\FR} \left((\mu_1, \Sigma_1), (\mu_2, \Sigma_2)\right)$ is a function of $\left(\Sigma_2^{-1/2}(\mu_1 - \mu_2), \Sigma_2^{-1/2}\Sigma_1^{1/2} \right)$. 
\end{Prop}

Now we replace the parameter space with $\mathbb{R}^d \times \GL(d,\mathbb{R})$. 
The tangent space $T_{(\mu,V)}(\mathbb{R}^d \times \GL(d,\mathbb{R}))$ is $\mathbb{R}^d \times \textup{M}(d,\mathbb{R})$ where $\textup{M}(d,\mathbb{R})$ denotes the space of real matrices of size $d\times d$. 
Define a map $\tau \colon \mathbb{R}^d \times \GL(d,\mathbb{R}) \to \mathbb{R}^d \times \SPD(d,\mathbb{R})$ by $\tau(\mu, V) \coloneqq  (\mu, V V^{\top})$. 

The derivative of $\tau$ satisfies that 
\[ d\tau_{(\mu,V)}(u,U) = \left(u, U V^{\top} + V U^{\top}\right), \ (u,U) \in \mathbb{R}^d \times \textup{M}(d,\mathbb{R}). \]
Hence, $\Ker(d\tau_{(\mu,V)}) = \left\{(0,VK) \colon K + K^{\top} = 0 \right\}$. 
Therefore, this map is not an immersion if $d \ge 2$.

Denote the pull-back of the metric $g_{(\mu,\Sigma)}$ with respect to $\tau$ by $\bar{g}_{(\mu,V)}$. 
Denote the {\it symmetrization} of a square matrix $A$ by $\sym(A)$, specifically, 
$\sym(A) \coloneqq  \frac{1}{2} (A + A^{\top})$.
We see that 
\[ \bar{g}_{(\mu,V)}\left((u,U), (w,W) \right) \]
\[ = \frac{4a_1}{d} \tilde{u}^{\top} \tilde{w} + \frac{8 a_2}{d(d+2)} \textup{tr}\left(\widetilde{U} \widetilde{W}\right) + \left(\frac{4a_2}{d(d+2)}  - 1 \right) \textup{tr}\left(\widetilde{U}\right) \textup{tr}\left(\widetilde{W}\right), \]
where we let $\tilde{u} \coloneqq  V^{-1} u$, $\tilde{w} \coloneqq  V^{-1} w$, $\widetilde{U} \coloneqq  \sym(V^{-1}U)$, and $\widetilde{W} \coloneqq  \sym(V^{-1}W)$. 
We see that for every $K$ with $K + K^{\top} = O$, 
\[ \bar{g}_{(\mu,V)}\left((0,VK),\cdot \right) = 0. \]

Let $\Phi_{(A,b)}(\mu, V) \coloneqq  (A\mu + b, AV)$. 
Then 
\[ d\Phi_{(A,b)}(u,U) = (Au, AU), \  (u,U) \in \mathbb{R}^d \times \textup{M}(d,\mathbb{R}). \]
Hence,  
\[ \bar{g}_{\Phi_{(A,b)}(\mu,V)}\left( d\Phi_{(A,b)}(u,U), d\Phi_{(A,b)}(w,W)\right) = \bar{g}_{(\mu,V)}\left((u,U), (w,W) \right). \]

Finally, we replace the parameter space with $\mathbb{R}^d \times (\GL(d,\mathbb{R})/\O(d,\mathbb{R}))$. 
Define a map $\tilde\tau \colon \mathbb{R}^d \times (\GL(d,\mathbb{R})/\O(d,\mathbb{R})) \to \mathbb{R}^d \times \SPD(d)$ by $\tilde\tau(\mu, [V]) \coloneqq  (\mu, V V^{\top})$. 
Let $\tilde{g}$ be the pull-back of the metric $g$ by the map $\tilde{\tau}$. 
This metric does not degenerate since $\tilde\tau$ is a diffeomorphism.

Let $q \colon \mathbb{R}^d \times \GL(d, \mathbb R) \to \mathbb{R}^d \times (\GL(d,\mathbb{R})/\O(d,\mathbb{R}))$ be the quotient map, specifically, $q(\mu,V) \coloneqq  (\mu, [V])$. 
Let $dq$ be the differential of the map $q$.

We remark that $dq_{(\mu,V)}$ is surjective and $\Ker (dq_{(\mu,V)})$ is not trivial. 
Indeed, 
\[ \Ker (dq_{(\mu,V)}) = T_{(\mu,V)}\{(\mu,VQ) \colon Q \in \O(d,\mathbb{R}) \} = \{(0,VK) \colon K + K^{\top} = O\}.\]  
For ease of notation, let $M \coloneqq  \mathbb{R}^d \times \GL(d, \mathbb R)$ and $\widetilde{M} \coloneqq  \mathbb{R}^d \times (\GL(d,\mathbb{R})/\O(d,\mathbb{R}))$. 
Then 
\[ T_{(\mu,[V])} \widetilde{M} \simeq T_{(\mu,V)} M / \Ker (dq_{(\mu,V)}) \simeq \mathbb{R}^d \times \textup{M}(d,\mathbb{R})/\{(0,VK) \colon K + K^{\top} = O\}.\]  

It holds that 
\[ UV^{\top} + V U^{\top} = 2V\sym(V^{-1}U)V^{\top}.\]
By the linear map $(u,U) \mapsto (u, \sym(V^{-1}U))$, 
\[ \mathbb{R}^d \times \textup{M}(d,\mathbb{R})/\{(0,VK) \colon K + K^{\top} = O\} \simeq \mathbb{R}^d \times \Sym(d). \]

Now we identify $T_{(\mu,[V])} \widetilde{M}$ with $\mathbb{R}^d \times \Sym(d)$. 
By this identification $dq_{(\mu,V)}(u,U)$ corresponds to $(u,\sym(V^{-1}U))$ and we obtain that 
\[ \tilde{g}_{(\mu,[V])}\left(dq_{(\mu,V)}(u,U), dq_{(\mu,V)}(w,W) \right) \]
\[ = g_{(\mu, V V^{\top})} \left(d\tilde{\tau}_{(\mu,[V])} \circ dq_{(\mu,V)}(u,U), d\tilde{\tau}_{(\mu,[V])} \circ dq_{(\mu,V)}(w,W) \right) \]
\[ = g_{(\mu, V V^{\top})} \left(d\tau_{(\mu,V)}(u,U), d\tau_{(\mu,V)}(w,W) \right) \]
\[ = g_{(\mu, V V^{\top})} \left((u,UV^{\top}+VU^{\top}), (w,WV^{\top}+VW^{\top}) \right) \]
\[ = \frac{4a_1}{d} \tilde{u}^{\top} \tilde{w} + \frac{8 a_2}{d(d+2)} \textup{tr}\left(\widetilde{U} \widetilde{W}\right) + \left(\frac{4a_2}{d(d+2)}  - 1 \right) \textup{tr}\left(\widetilde{U}\right) \textup{tr}\left(\widetilde{W}\right). \]
Hence, the pull-back of the metric $\tilde{g}$ by the quotient map $q$ is identical to $\bar{g}$.

We remark that $\tilde{g}$ does {\it not} degenerate. 

The following is a commutative diagram:
\[
\begin{tikzcd}
{\mathbb{R}^d \times \GL(d,\mathbb{R})}
\arrow[d, "{q}"']
\arrow[dr, "{\tau}"] & \\
{\mathbb{R}^d \times (\GL(d,\mathbb{R})/O(d,\mathbb{R}))}
\arrow[r, "{\tilde\tau}"'] &
{\mathbb{R}^d \times \SPD(d)}
\end{tikzcd}
\]

We have 
\begin{Prop}\label{prop:FR-quotient}
If the parameter space is $\mathbb{R}^d \times (\GL(d, \mathbb{R})/\O(d, \mathbb{R}))$, then we have the following:\\
(i) 
\[ d_{\FR}\left((A\mu_1 + b, [AV_1]), (A\mu_2 + b,  [A V_2]) \right) = d_{\FR} \left((\mu_1, [V_1]), (\mu_2, [V_2])\right). \]
(ii) 
$d_{\FR} \left((\mu_1, [V_1]), (\mu_2, [V_2])\right)$ is a function of 
the singular values and their multiplicities of $V_2^{-1}V_1$ and the block norms  determined by the pair $(V_2^{-1}V_1, V_2^{-1}(\mu_1 - \mu_2))$.
\end{Prop}

\begin{Exa}\label{exa:FR-Gaussian}
For the location-scale model of the multivariate normal distribution, 
$$f_1 (r) = \frac{1}{(\sqrt{2\pi})^d} \exp\left(- \frac{r}{2} \right)$$ and hence $\displaystyle \frac{f_1^{\prime} (r)}{f_1 (r)} = -\frac{1}{2}$. 
Since the expectation and the variance of the chi-squared distribution of $d$ degree of freedom  are $d$ and $2d$, respectively, 
we see that $a_1 = \dfrac{d}{4}$ and $a_2 = \dfrac{d(d+2)}{4}$. 
Therefore, we obtain that 
\[ g_{(\mu, \Sigma)} \left( (u,U), (w,W)\right)  = u^{\top} \Sigma^{-1} w + \frac{1}{2} \textup{tr}(\Sigma^{-1}U\Sigma^{-1}W). \]
\end{Exa}

\begin{Rem}
The statement in Example \ref{exa:FR-Gaussian} should be understood as referring to the general two-point boundary-value problem rather than the geodesic initial-value problem. 
For the multivariate normal model,
explicit geodesic parametrizations are known, 
while a closed-form expression for the general Fisher--Rao distance remains unknown for  $d>1$~\cite{pinele2020fisher,nielsen2024approximation}. 

By Proposition \ref{prop:FR-quotient} (i), 
if $\nu \coloneqq V_2^{-1}(\mu_1-\mu_2)$ and $S\coloneqq V_2^{-1}V_1$,
then
\[ d_{\FR}\left((\mu_1,[V_1]),(\mu_2,[V_2])\right) = d_{\FR}\left((\nu,[S]),(0,[I_d])\right). \]
Via the diffeomorphism
$\tilde\tau(\mu,[V]) = \left(\mu,VV^\top \right)$,
this reduces the problem to the pair
$ \left((\nu,SS^\top),(0,I_{d})\right)$ in $\mathbb{R}^d \times \SPD(d)$ endowed with the Gaussian Fisher metric in
Example \ref{exa:FR-Gaussian}.

For this metric, Calvo and Oller \cite[Theorem 3.1 and Section 4]{calvo1991explicit} obtained an explicit representation of geodesics starting from $(0,I_d)$. 
Later, Kobayashi~\cite{kobayashi2023geodesics} described the geodesics between any two multivariate normal distributions.
\end{Rem}

\begin{Rem}
Lovri\'{c}, Min-Oo and Ruh \cite{lovric2000multivariate} proposed another invariant Riemannian geometry on the manifold of multivariate normal distributions. 
Their construction identifies the Gaussian location-scale manifold with the Riemannian symmetric space $\SL(d+1,\mathbb{R})/\SO(d+1,\mathbb{R})$. 

In the notation of this paper, for $\theta = (\mu,[V]) \in \mathbb{R}^d \times \GL(d,\mathbb{R})/\O(d,\mathbb{R})$, 
define 
\[ \iota (\mu,[V]) \coloneqq |\det V|^{-2/(d+1)} \begin{pmatrix} V V^{\top} + \mu \mu^{\top} & \mu \\ \mu^{\top} & 1 \end{pmatrix}. \]
The map $\iota$ is a diffeomorphism between $\mathbb{R}^d \times \GL(d,\mathbb{R})/\O(d,\mathbb{R})$ and $\SL(d+1,\mathbb{R})/\SO(d+1,\mathbb{R})$. 
Assume that $A \in \GL(d,\mathbb{R})$ and $\det A > 0$ and $b \in \mathbb{R}^d$. 
Let 
\[ j(A,b) \coloneqq (\det A)^{-1/(d+1)} \begin{pmatrix} A & b \\ 0 & 1 \end{pmatrix} \in \SL(d+1,\mathbb{R}). \]
Then 
\[ \iota(A\mu + b, [AV]) = j(A,b) \iota(\mu, [V]) j(A,b)^{\top}.\]
The map $j$ is the special-linear realization of the orientation-preserving affine group.
By the map $\iota$, 
the affine action on the Gaussian location-scale family becomes the restriction of the $\SL(d+1,\mathbb{R})$-congruence action on $\SL(d+1,\mathbb{R})/\SO(d+1,\mathbb{R})$. 

The induced metric on $\SL(d+1,\mathbb{R})/\SO(d+1,\mathbb{R})$ is not the Fisher metric; rather, it is a different
invariant metric for which the geodesic distance can be expressed explicitly in terms of the eigenvalues of the positive definite matrix representing the relative position of $\theta_1 = (\mu_1, [V_1])$ with respect to $\theta_2 = (\mu_2, [V_2])$ with determinant one. 
Hence, this construction should be viewed as a symmetric-space analogue of the affine-invariant Fisher--Rao geometry discussed above, not as a closed-form expression for the Fisher--Rao distance.
\end{Rem}

\bibliographystyle{plain}
\bibliography{maxinv-div-dist}

\end{document}